\documentclass[11pt,a4paper,reqno]{amsart}
\usepackage[utf8]{inputenc}
\usepackage[T1]{fontenc}
\usepackage[english]{babel}

\usepackage{amsmath,xcolor,color,dsfont}
\usepackage{amsfonts}
\usepackage{amssymb,amsthm}

\colorlet{darkgreen}{green!50!black}

\usepackage{hyperref} % Créer des liens et des signets 
\hypersetup{	
colorlinks=true, %colorise les liens 
breaklinks=true, %permet le retour à la ligne dans les liens trop longs 
urlcolor= green, %couleur des hyperliens 
linkcolor= blue,	%couleur des liens internes 
citecolor=darkgreen,	%couleur des références 
pdftitle={Tutte's invariant approach for Brownian motion reflected in the quadrant}, %informations apparaissant dans 
pdfauthor={Sandro Franceschi et Kilian Raschel}, %les informations du document 
pdfsubject={Probability}	%sous Acrobat. 
}

\newtheorem{thm}{Theorem}
\newtheorem{prop}[thm]{Proposition}

\newtheorem{lem}[thm]{Lemma}

\renewcommand{\geq}{\geqslant}
\renewcommand{\leq}{\leqslant}

\usepackage{epsfig}
\usepackage{dsfont}

\usepackage{geometry}
\usepackage{enumitem}
\usepackage{mathtools}
\usepackage{eqnarray,cases}

\usepackage{indentfirst}

\usepackage{pstricks-add}
\usepackage[numbers]{natbib}

\newcommand{\citeep}[1]{\citeeauthor*{#1} (\citeeyear{#1}) \citeep{#1}}

\usepackage{float}

\newcommand{\s}{\mathcal{S}}

\newcommand{\G}{\mathcal G_\mathcal R}
\newcommand{\R}{\mathcal{R}}
\let\phi=\varphi

\title[Tutte's invariant approach for Brownian motion reflected in the quadrant]{Tutte's invariant approach for Brownian motion reflected in the quadrant}

\author{S.\ Franceschi} \address{Laboratoire de Probabilit\'es et
        Mod\`eles Al\'eatoires, Universit\'e Pierre et Marie Curie,
        4 Place Jussieu, 75252 Paris Cedex 05, France
        \& Laboratoire de Math\'ematiques et Physique Th\'eorique, Universit\'e de Tours, Parc de Grandmont, 37200 Tours, France} \email{sandro.franceschi@upmc.fr}

\author{K.\ Raschel} \address{CNRS \& F\'ed\'eration de recherche Denis Poisson \& Laboratoire de Math\'ematiques et Physique Th\'eorique, Universit\'e de Tours, Parc de Grandmont, 37200 Tours, France} \email{Kilian.Raschel@lmpt.univ-tours.fr}

\begin{document}
\maketitle

\date{\today}

\selectlanguage{english}
{
\begin{abstract}
We consider a Brownian motion with negative drift in the quarter plane with orthogonal reflection on the axes. The Laplace transform of its stationary distribution satisfies a functional equation, which is reminiscent from equations arising in the enumeration of (discrete) quadrant walks. We develop a Tutte's invariant approach to this continuous setting, and we obtain an explicit formula for the Laplace transform in terms of generalized Chebyshev polynomials.
\end{abstract}
}

\footnote{Version of \today}

\selectlanguage{english}

\section{Introduction and main results}
\label{sec:introduction}

\subsection{Reflected Brownian motion in the quadrant}
The object of study here is the reflected Brownian motion with drift in the quarter plane
\begin{equation}
\label{eq:RBMQP}
     Z(t)=Z_0 + W(t) + \mu t + R L(t),\qquad \forall t\geq 0,
\end{equation}
associated to the triplet $(\Sigma, \mu, R)$, composed of a non-singular covariance matrix, a drift and a reflection matrix, see Figure~\ref{fig:drift_reflection}:
\begin{equation*}
     \Sigma = \left(  \begin{array}{ll} \sigma_{11} & \sigma_{12} \\ \sigma_{12} & \sigma_{22} \end{array} \right),\qquad \mu= \left(  \begin{array}{l} \mu_1 \\  \mu_2  \end{array} \right), \qquad  R= (R^1,R^2)=\left(  \begin{array}{cc} r_{11} & r_{12} \\ r_{21} & r_{22} \end{array} \right).
\end{equation*}     
In Equation~\eqref{eq:RBMQP}, $Z_0$ is any initial point in $\mathbb R_+^2$, the process $(W(t))_{t \geq 0} $ is an unconstrained planar Brownian motion with covariance matrix $\Sigma$ starting from the origin, and for $i=1,2$, $L^i(t)$ is a continuous non-decreasing process, that increases only at time $t$ such that $Z^i(t)=0$, namely $\int_{0}^t \mathds{1}_{\{Z^i(s) \ne 0 \}} \mathrm{d} L^i(s)=0$, for all $t\geq 0$. The columns $R^1$ and $R^2$ represent the directions in which the Brownian motion is pushed when the axes are reached.

\begin{figure}[hbtp]
\centering
\includegraphics[scale=1]{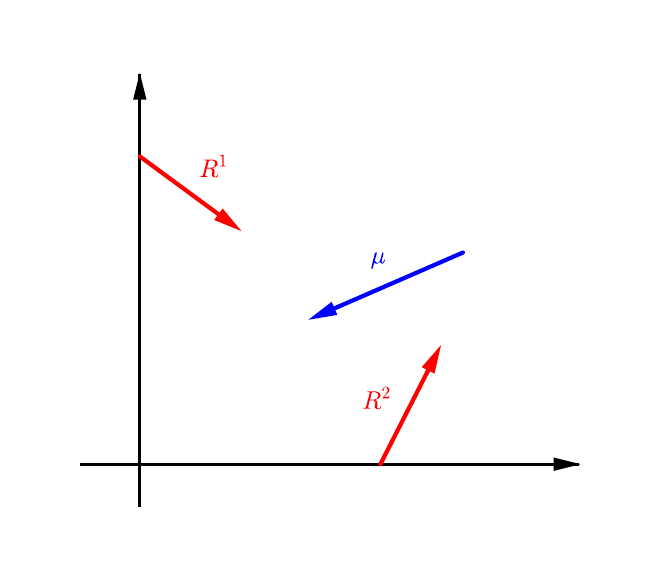}
\includegraphics[scale=1]{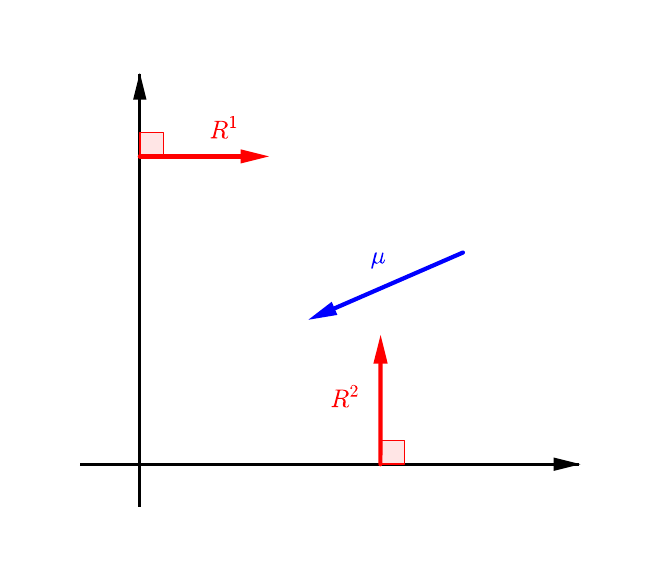}
\caption{Drift $\mu$ and reflection vectors $R^1$ and $R^2$ in non-orthogonal and orthogonal cases}
\label{fig:drift_reflection}
\end{figure}

The reflected Brownian motion $(Z(t))_{t\geq0}$ associated with $(\Sigma, \mu, R)$ is well defined \cite{harrison_multidimensional_1987,williams_semimartingale_1995}, and is a fondamental stochastic process, from many respects. There is a large literature on reflected Brownian motion in the quadrant (and also in orthants, generalization to higher dimension of the quadrant). First, it serves as an approximation of large queuing networks \cite{foddy_1984,Baccelli_Fayolle_1987}; this was the initial motivation for its study. In the same vein, it is the continuous counterpart of (random) walks in the quarter plane, which are an important combinatorial and probabilistic object, see \cite{BMMi-10,BeBMRa_2015}. In other directions, it is studied for its Lyapunov functions \cite{DuWi-94},  cone points of Brownian motion \cite{lega-87}, intertwining relations and crossing probabilities \cite{Du-04}, and of particular interest for us, for its recurrence or transience \cite{hobson_recurrence_1993}. Its stationary distribution exists and is unique if and only if the following (geometric) conditions are satisfied: 
\begin{equation}
\label{eq:CNS_ergodic}
     r_{11} > 0, \ \  r_{22} > 0, \ \ r_{11} r_{22} - r_{12} r_{21} > 0,\ \ r_{22} \mu_1 - r_{12}  \mu_2 < 0, \ \ r_{11} \mu_2 - r_{21}  \mu_1 < 0. 
\end{equation}
(For orthogonal reflections, \eqref{eq:CNS_ergodic} is equivalent for the drift $\mu$ to have two negative coordinates.) Moreover, the asymptotics of the stationary distribution (when it exists) is now well known, see \cite{harrison_reflected_2009,dai_reflecting_2011,franceschi_asymptotic_2016}. 

There exist, however, very few results giving an exact expression for the stationary distribution, and the main contribution of this paper is precisely to propose a method (based on boundary value problems) for deriving an explicit formula for the (Laplace transform of the) stationary distribution. Our study constitutes one of the first attempts to apply these techniques to reflected Brownian motion, after \cite{Foschini} (under the strong symmetry condition $\mu_1=\mu_2$, $\sigma_{11}=\sigma_{22}$ and symmetric reflection vectors $R^1,R^2$ in~\eqref{eq:RBMQP}), \cite{foddy_1984} (with the identity covariance matrix $\Sigma$) and \cite{Baccelli_Fayolle_1987} (with a diffusion having a quite special behavior on the boundary). We also refer to \cite{BuChMaRa_2015} for the analysis of reflected Brownian motion in bounded domains by complex analysis techniques.

\subsection{Laplace transform of the stationary distribution}
Under assumption~\eqref{eq:CNS_ergodic}, that we shall do throughout the manuscript, the stationary distribution is absolutely continuous w.r.t.\ the Lebesgue measure, see \cite{harrison_multidimensional_1987,dai_steady-state_1990}. We denote its density by $\pi(x)=\pi(x_1,x_2)$. Let the moment generating function (or Laplace transform) of $\pi$ be defined by
\begin{equation*}
     \varphi (\theta) = \mathbb{E}_{\pi} [e^{ \langle \theta \vert  Z\rangle}] = \iint_{{\mathbb R}_+^2} e^{\langle \theta \vert x \rangle} \pi(x) \mathrm{d} x.
\end{equation*}
The above integral converges at least for $\theta=(\theta_1,\theta_2) \in \mathbb{C}^2$ such that ${\Re}\,\theta_1\leq 0$ and ${\Re}\,\theta_2\leq 0$.
We further define two finite boundary measures $\nu_1$ and $\nu_2$ with support on the axes, by mean of the formula
\begin{equation*}
     \nu_i (B) = \mathbb{E}_{\pi} \bigg[ \int_0^1 \mathds{1}_{\{Z(t) \in B\}} \mathrm{d}L^i (t)\bigg].
\end{equation*}     
The measure $\nu_i$ is continuous w.r.t.\ the Lebesgue measure \cite{harrison_multidimensional_1987}, and may be viewed as the boundary invariant measure. We define the moment generating function of $\nu_1$ and $\nu_2$ by
\begin{equation*}
\varphi_2 (\theta_1) 
%=
%\mathbb{E}_{\Pi} [ \int_0^1 e^{\theta_1 Z_t^1}  \mathrm{d} L^2(t)]
=\int_{{\mathbb R}_+} e^{\theta_1 x_1} \nu_2(x_1) \mathrm{d} x_1,
\qquad
\varphi_1 (\theta_2) 
%=
%\mathbb{E}_{\Pi} [ \int_0^1 e^{\theta_2 Z_t^2}  \mathrm{d} L^1(t)]
=\int_{{\mathbb R}_+} e^{\theta_2 x_2} \nu_1(x_2) \mathrm{d} x_2.
\end{equation*}
The functions $\phi_1$ and $\phi_2$ exist {\it a priori} for values of the argument with non-positive real parts. There is a functional equation between the Laplace transforms $\phi$, $\phi_1$ and $\phi_2$, see~\eqref{eq:functional_equation} in Section~\ref{sec:analytic}, which is reminiscent of the functional equation counting (discrete) quadrant walks \cite{BMMi-10,BeBMRa_2015}.
    
\subsection{Main result}
We derive an explicit expression for $\phi_1$, and therefore also for $\phi_2$ and $\phi$ by the functional equation~\eqref{eq:functional_equation}, in the particular case where $R$ is the identity matrix, which means that the reflections are orthogonal (Figure~\ref{fig:drift_reflection}, right). Define the generalized Chebyshev polynomial by (for $a\geq 0$)
\begin{equation*}
T_a(x)  =\cos (a\arccos (x))=\frac{1}{2} \Big\{\big(x+\sqrt{x^2-1}\big)^a+\big(x-\sqrt{x^2-1}\big)^a\Big\}.
\end{equation*}
It admits an analytic continuation on $\mathbb C\setminus (-\infty,-1)$, and even on $\mathbb C$ if $a$ is a non-negative integer. % \textcolor{red}{en fait, comme je l'avais déjà remplacé, j'enlève plutôt l'intervalle (-infty,-1), en effet c'est cohérent avec la détermination principale du log qu'il vaut mieux prendre pour tout l'article: $\arccos (x)= -i\ln (x+\sqrt{x^2-1})$ a sa "détermination principale" qui découle du log sur $\mathbb C\setminus (-\infty,-1)$, c'est pour ça que j'avais mis des "moins" partout après dans $w(s)=(-s)^{...}$ et dans $w(\theta_2)=T_{..}(- ..)$} 
We also need to introduce 
\begin{equation}
\label{eq:definition_theta_2_pm}
     \theta_2^\pm = \frac{(\mu_1\sigma_{12}-\mu_2\sigma_{11}) \pm \sqrt{(\mu_1\sigma_{12}-\mu_2\sigma_{11})^2 +\mu_1^2 \det{\Sigma}}}{\det{\Sigma}} 
\end{equation}
(notice that the sign of $\theta_2^\pm$ is $\pm$, see Figure~\ref{BVPtheta}), as well as the angle (related to the correlation coefficient of the Brownian motion $(W(t))_{t\geq0}$)
\begin{equation}
\label{eq:definition_beta}
     \beta= \arccos \left({-\frac{\sigma_{12}}{\sqrt{\sigma_{11}\sigma_{22}}}}\right).
\end{equation}

\begin{thm} 
\label{thm:main}
Let $R$ be the identity matrix in~\eqref{eq:RBMQP}. The Laplace transform $\phi_1$ is equal to 
\begin{equation*}
     \phi_1 (\theta_2)= \frac{-\mu_1 {w}'(0)}{{w}(\theta_2)-{w}(0)}\theta_2,
\end{equation*}
where the function $w$ can be expressed in terms of the generalized Chebyshev polynomial $T_{\frac{\pi}{\beta}}$ as follows:
\begin{equation*}
     {w} (\theta_2)=T_{\frac{\pi}{\beta}}\bigg(-\frac{2\theta_2-(\theta_2^++\theta_2^-)}{\theta_2^+-\theta_2^-}\bigg).
\end{equation*}
Accordingly, $\phi_1$ can be continued meromorphically on the cut plane $\mathbb{C}\setminus (\theta_2^+,\infty)$.
\end{thm}
There exists, of course, an analogous expression for $\phi_2(\theta_1)$, and the functional equation~\eqref{eq:functional_equation} finally gives a simple explicit formula for the bivariate Laplace transform $\phi$.

Let us now give some comments around Theorem~\ref{thm:main}.
\begin{itemize}
     \item It connects two {\it a priori} unrelated objects: the stationary distribution of reflected Brownian motion in the quadrant and a particular special function, viz, a generalized Chebyshev polynomial (which is a hypergeometric function). The expression that we obtain is quite tractable: as an example we will recover the well-known case of one-dimensional reflected Brownian motion (Section~\ref{subsec:diagonal_covariance}).
     \item To prove Theorem~\ref{thm:main}, we apply a constructive (and combinatorial in nature) variation of the boundary value method of \cite{fayolle_random_1999}, recently introduced in \cite{BeBMRa_2015} as Tutte's invariant approach \cite{Tutte-95}, see our Section~\ref{sec:BVP}. This paper is one of the first attempts to apply boundary value techniques to (continuous) diffusions in the quadrant, after \cite{foddy_1984} (which concerns very particular cases of the covariance matrix, essentially the identity matrix) and \cite{Baccelli_Fayolle_1987} (on diffusions with completely different behavior on the boundary). 
     \item The recent paper \cite{franceschi_asymptotic_2016} obtains the exact asymptotic behavior of the stationary distribution along any direction in the quadrant. The constant in that asymptotics involves both functions $\phi_1$ and $\phi_2$ (see \cite{franceschi_asymptotic_2016}), and can thus be made explicit with our Theorem~\ref{thm:main}.
     \item Theorem~\ref{thm:main} is well suited for asymptotic analysis: we shall derive the asymptotics of the stationary distribution along the axes, by using classical arguments from singularity analysis of Laplace transforms (see the reference book \cite{Doetsch}). See our Section~\ref{sec:singularity}.
     \item Theorem~\ref{thm:main} implies that the Laplace transform is algebraic if and only if a certain group (to be properly introduced later on) is finite, see Section~\ref{subsec:group}. This result has an analogue in the discrete setting, see \cite{BMMi-10}. In the same vein, the authors of \cite{DiMo-09} give necessary and sufficient conditions for the stationary density of two-dimensional reflected Brownian motion with negative drift in a wedge to have the form of a sum of exponentials. The intersection with our results is the polynomial case $\frac{\pi}{\beta}\in\mathbb Z$.
     \item More that the Brownian motion in the quadrant, all results presented here (including Theorem \ref{thm:main}) concern the Brownian motion in two-dimensional cones (by a simple linear transformation of the cones). This is a major and interesting difference of the continuous case in comparison with the discrete case, which also illustrates that the analytic approach is very well suited to that context. 
\end{itemize}

Though being self-contained (this is one of the reasons why we focus on the case of orthogonal reflections), this paper is part of a larger project, dealing with any reflection matrix $R$ (as in Figure~\ref{fig:drift_reflection}, left). Let us finally mention that an extended abstract of this paper and of \cite{franceschi_asymptotic_2016} may be found in \cite{EA-AofA2016}.

\subsection*{Acknowledgements}
We thank Mireille Bousquet-M\'elou, Irina Kurkova and Marni Mishna for interesting discussions. We also thank an anonymous referee for her/his careful reading and her/his suggestions. Finally, we acknowledge support from the projet R\'egion Centre-Val de Loire (France) MADACA and from Simon Fraser University, Burnaby, BC (Canada).

\section{Analytic preliminaries and continuation of the Laplace transforms}
\label{sec:analytic}

In this section we state the key functional equation (a kernel equation, see Section~\ref{subsec:functional_equation}), which is the starting point of our entire analysis. We study the kernel (a second degree polynomial in two variables) in Section~\ref{subsec:kernel}. Finally, we continue the Laplace transforms to larger domains (Section~\ref{subsec:continuation}), which will be used in Section~\ref{sec:BVP} to state a boundary value problem (BVP).

\subsection{A kernel functional equation}
\label{subsec:functional_equation}
We have the following key functional equation between the Laplace transforms:
\begin{equation}  
\label{eq:functional_equation}
     -\gamma (\theta) \varphi (\theta) =\gamma_1 (\theta) \varphi_1 (\theta_2) + \gamma_2 (\theta) \varphi_2 (\theta_1),
\end{equation}
where
\begin{align*}
  \begin{cases}
     \phantom{{}_1}\gamma (\theta)= \frac{1}{2} \langle \theta \vert  \sigma \theta \rangle+ \langle \theta \vert  \mu \rangle =\frac{1}{2}(\sigma_{11} \theta_1^2 + \sigma_{22} \theta_2^2 +2\sigma_{12}\theta_1\theta_2)
+
\mu_1\theta_1+\mu_2\theta_2,  \\
     \gamma_1 (\theta)= \langle R^1 \vert  \theta \rangle=r_{11} \theta_1 + r_{21} \theta_2,  \\
     \gamma_2 (\theta)=\langle R^2 \vert  \theta \rangle=r_{12} \theta_1 + r_{22} \theta_2.
  \end{cases}
\end{align*}
By definition of the Laplace transforms, this equation holds at least for any $\theta=(\theta_1, \theta_2)$ with ${\Re}\,\theta_1\leq 0$ and ${\Re}\, \theta_2\leq 0$.  

To prove \eqref{eq:functional_equation} the main idea is to use an identity called basic adjoint relationship (first proved in \cite{harrison_brownian_1987} in some particular cases, then extended in \cite{dai_reflected_1992}), which characterizes the stationary distribution, see \cite{dai_characterization_1994,dai_characterization_2010,dai_nonnegativity_2011}. (It is the continuous analogue of the equation $\pi Q=0$, with $\pi$ the stationary distribution of a recurrent continuous-time Markov chain having infinitesimal generator matrix $Q$.) This basic adjoint relationship connects the stationary distribution $\pi$ and the corresponding boundary measures $\nu_1$ and $\nu_2$. %Our functional equation is obtained taking $f$ exponential in order to make Laplace transforms appear. 
We refer to \cite{foddy_1984,dai_reflected_1992,dai_reflecting_2011} for the details.

As stated in \cite[Open problem 1]{dai_nonnegativity_2011}, it is an open problem to prove that there is no finite \textit{signed} measure which satisfies the basic adjoint relationship (by signed we mean a measure taking positive and negative values). Going back to our case, we shall see that the functional equation \eqref{eq:functional_equation} admits a unique solution (the stationary distribution); in particular there is no solution of the functional equation with sign changes.

\subsection{Kernel}
\label{subsec:kernel}
By definition, the kernel of Equation~\eqref{eq:functional_equation} is the polynomial
\begin{equation*}  
     \gamma(\theta_1,\theta_2)=\frac{1}{2}(\sigma_{11} \theta_1^2 +2\sigma_{12}\theta_1\theta_2+ \sigma_{22} \theta_2^2 )+\mu_1\theta_1+\mu_2\theta_2.
\end{equation*}  
It can be alternatively written as
\begin{equation*}  
     \gamma(\theta_1, \theta_2)= \widetilde{a}(\theta_2)\theta_1^2+\widetilde{b}(\theta_2)\theta_1+\widetilde{c}(\theta_2)=a(\theta_1)\theta_2^2+b(\theta_1)\theta_2+c(\theta_1),
\end{equation*}  
where
\begin{equation}
\label{eq:definition_a_b_c}
\left\{\begin{array}{lll}
\widetilde{a}(\theta_2)=\frac{1}{2}\sigma_{11},\quad & \quad\widetilde{b}(\theta_2)=\sigma_{12}\theta_2+\mu_1,\quad & \quad\widetilde{c}(\theta_2)=\frac{1}{2}\sigma_{22}\theta_2^2+\mu_2\theta_2,\smallskip\smallskip \\ 
a(\theta_1)=\frac{1}{2}\sigma_{22},\quad & \quad b(\theta_1)=\sigma_{12}\theta_1+\mu_2,\quad & \quad c(\theta_1)=\frac{1}{2}\sigma_{11}\theta_1^2+\mu_1\theta_1. \\ 
\end{array} \right.
\end{equation}
The equation $\gamma(\theta_1, \theta_2) = 0$ defines a two-valued algebraic function $\Theta_1(\theta_2)$  such that $\gamma(\Theta_1(\theta_2), \theta_2)= 0$, and similarly $\Theta_2(\theta_1)$ such that $\gamma(\theta_1, \Theta_2(\theta_1))= 0$. Their expressions are given by
\begin{equation*}
     \Theta_1^\pm(\theta_2)=\frac{-\widetilde{b}(\theta_2)\pm\sqrt{\widetilde{d}(\theta_2)}}{2\widetilde{a}(\theta_2)},\qquad \Theta_2^\pm(\theta_1)=\frac{-b(\theta_1)\pm\sqrt{d(\theta_1)}}{2a(\theta_1)},
\end{equation*}
where $\widetilde d$ and $d$ are the discriminants of the kernel:
\begin{equation*}
\left\{\begin{array}{l}
\widetilde{d}(\theta_2)=\theta_2^2(\sigma_{12}^2-\sigma_{11}\sigma_{22})+2\theta_2(\mu_1\sigma_{12}-\mu_2\sigma_{11})+\mu_1^2,\smallskip
 \\ d(\theta_1)=\theta_1^2(\sigma_{12}^2-\sigma_{11}\sigma_{22})+2\theta_1(\mu_2\sigma_{12}-\mu_1\sigma_{22})+\mu_2^2. 
 \end{array} \right.
\end{equation*}
The polynomials $\widetilde d$ and $d$ have two zeros, real and of opposite signs, they are denoted by $\theta_2^{\pm}$ and $\theta_1^\pm$, respectively: $\theta_2^\pm$ is introduced in~\eqref{eq:definition_theta_2_pm} and
\begin{equation*}
     \theta_1^\pm= \frac{(\mu_2\sigma_{12}-\mu_1\sigma_{22}) \pm \sqrt{(\mu_2\sigma_{12}-\mu_1\sigma_{22})^2 +\mu_2^2 \det{\Sigma}}}{\det{\Sigma}}.
\end{equation*}
Equivalently, $\theta_1^\pm$ and $\theta_2^\pm$ are the branch points of the algebraic functions $\Theta_2$ and $\Theta_1$.

Finally, notice that $d$ is positive on $(\theta_1^-, \theta_1^+)$ and negative on ${\mathbb R} \setminus [\theta_1^-, \theta_1^+]$. Accordingly, the branches $\Theta_2^\pm$ take real and complex conjugate values on these sets, respectively. A similar statement holds for $\Theta_1^\pm$.

%\begin{figure}[hbtp]
%\centering
%\includegraphics[scale=0.6]{dessinarticle/ellipsefonctions2.pdf}
%\includegraphics[scale=0.6]{dessinarticle/ellipsefonctions1.pdf}
%\caption{Functions $\Theta_2^\pm (\theta_1)$ and $\Theta_1^\pm (\theta_2)$ on $[\theta_1^-,\theta_1^+]$ and $[\theta_2^-,\theta_2^+]$.}
%\label{fonctionellipse}
%\end{figure} 

\subsection{Continuation of the Laplace transforms}
\label{subsec:continuation}
In Section~\ref{sec:BVP} we shall state a boundary condition for the functions $\phi_1$ and $\phi_2$, on curves which lie outside their natural domains of definition (the half-plane with negative real-part). We therefore need to continue these functions, which is done in the result hereafter.

\begin{lem}
\label{lem:prolongement}
We can continue meromorphically $\phi_1(\theta_2)$ to the open and simply connected set
\begin{equation}
\label{eq:domain_continuation}
     \{\theta_2\in\mathbb{C} : \Re\,\theta_2\leqslant 0 \text{ or } \Re\, \Theta_1^-(\theta_2) <0\},
\end{equation}
by setting
\begin{equation*}
     \phi_1(\theta_2)=\frac{\gamma_2}{\gamma_1}(\Theta_1^-(\theta_2),\theta_2)\phi_2(\Theta_1^-(\theta_2)).
\end{equation*}
\end{lem}

Lemma~\ref{lem:prolongement} is immediate (one can find a refined version of it in \cite{franceschi_asymptotic_2016}). A similar statement holds for the Laplace transform $\phi_1$. Anticipating slightly, we notice that the domain in~\eqref{eq:domain_continuation} contains the domain $\G$ in Figure~\ref{BVPtheta}.

\section{Statement of the boundary value problem}  
\label{sec:BVP}

\subsection{An important hyperbola}

For further use, we need to introduce the curve
\begin{equation}
\label{eq:curve_definition}
     \R =\{\theta_2\in\mathbb C: \gamma(\theta_1,\theta_2)=0 \text{ and } \theta_1\in(-\infty,\theta_1^-)\}=\Theta_2^\pm ((-\infty,\theta_1^-)).
\end{equation}
By the results of Section~\ref{subsec:kernel}, $d$ is negative on the interval $(-\infty,\theta_1^-)$, and thus the curve $\R$ is symmetrical w.r.t.\ the real axis, see Figure~\ref{BVPtheta}. Furthermore, it has a simple structure, as shown by the following elementary result, taken from \cite[Lemma 9]{Baccelli_Fayolle_1987}:
\begin{lem}%[\cite{Baccelli_Fayolle_1987}]
The curve $\R$ in~\eqref{eq:curve_definition} is a {\rm(}branch of a{\rm)} hyperbola, given by the equation
\begin{equation}
\label{eq:hyperbole}
     \sigma_{22}(\sigma_{12}^2-\sigma_{11}\sigma_{22})x^2+\sigma_{12}^2\sigma_{22}y^2-2\sigma_{22}(\sigma_{11}\mu_2-\sigma_{12}\mu_1)x=\mu_2(\sigma_{11}\mu_2-2\sigma_{12}\mu_1).
\end{equation}
\end{lem}
\begin{proof}
We saw in Section~\ref{subsec:kernel} that whenever $\theta_1\in(-\infty,\theta_1^-)$, the branches $\Theta_2^+(\theta_1)$ and $\Theta_2^-(\theta_1)$ are complex conjugate. Define $x$ and $y$ by $\Theta_2^+(\theta_1)=x+iy$. Then from the identities
\begin{equation*}
     \Theta_2^+ +\Theta_2^- = -\frac{b(\theta_1)}{a(\theta_1)}
=-2\frac{\sigma_{12}\theta_1+\mu_2}{\sigma_{22}}
=2x
\quad \text{and} \quad
\Theta_2^+ \cdot\Theta_2^- =\frac{c(\theta_1)}{a(\theta_1)}
=\frac{\sigma_{11}\theta_1^2+2\mu_1 \theta_1}{\sigma_{22}}
=x^2+y^2,
\end{equation*}
we obtain that $x$ and $y$ are real solutions of  \eqref{eq:hyperbole}.
\end{proof}

We shall denote by $\G$ the open domain of $\mathbb{C}$ bounded by $\R$ and containing $0$, see Figure~\ref{fig:correspondance_plane_sphere}. Obviously $\overline{\G}$, the closure of $\G$, is equal to $\G\cup\R$.

\begin{figure}[hbtp]
\centering
\includegraphics[scale=1]{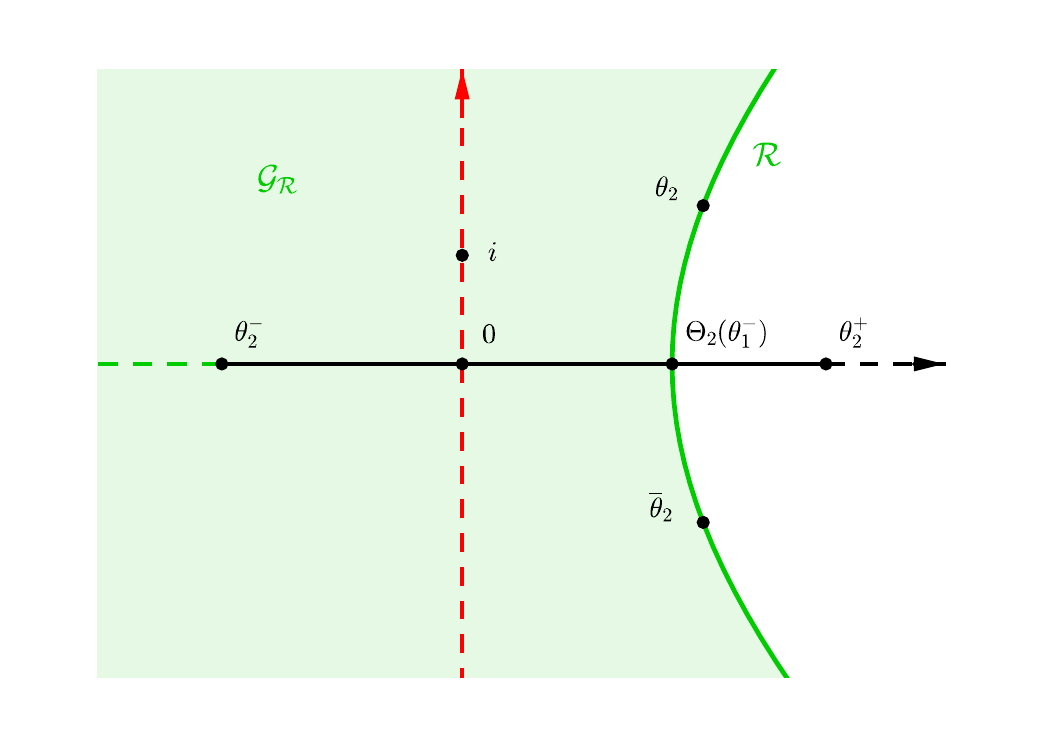}
\caption{The curve $\R$ in~\eqref{eq:curve_definition} is symmetric w.r.t.\ the horizontal axis, and $\G$ is the domain in green}
\label{BVPtheta}
\end{figure}

\subsection{BVP for orthogonal reflections}
\label{sec:boundary_cond}

In the case of orthogonal reflections, $R$ is the identity matrix in~\eqref{eq:RBMQP}, and we have $\gamma_1(\theta_1,\theta_2)=\theta_1$ and $\gamma_2 (\theta_1,\theta_2)=\theta_2$. We set
\begin{equation}
\label{eq:definition_psi}
     \psi_1 (\theta_2)=\frac{1}{\theta_2}\varphi_1 (\theta_2),\qquad \psi_2 (\theta_1)=\frac{1}{\theta_1}\varphi_2 (\theta_1).
\end{equation}

\begin{prop}
\label{prop:BVP}
The function $\psi_1$ satisfies the following BVP:
\begin{enumerate}[label={\rm (\roman{*})},ref={\rm (\roman{*})}]
     \item\label{item:1_BVP} $\psi_1$ is meromorphic on $\G$ with a single pole at $0$, of order $1$ and residue $\varphi_1(0)$, and vanishes at infinity,
     \item\label{item:2_BVP} $\psi_1$ is continuous on $\overline{\G}\setminus \{0\}$ and
\begin{equation}
\label{eq:boundary_condition}
     \psi_1(\overline{\theta_2})=\psi_1({\theta_2}), \qquad \forall \theta_2\in \R.
\end{equation}
\end{enumerate}
\end{prop}

\begin{proof}
Using the formula~\eqref{eq:definition_psi}, Point~\ref{item:1_BVP} is equivalent to proving that $\phi_1$ is analytic in $\G$ and is bounded at infinity. Both properties are obvious in the half-plane $\{\theta_2\in\mathbb C: \Re\,\theta_2<0\}$, thanks to the definition of $\phi_1$ as a Laplace transform. In the domain of $\G$ where $\Re\,\theta_2\geq0$, one can use the continuation formula given in Lemma~\ref{lem:prolongement}, and the conclusion follows.

Let us now prove~\ref{item:2_BVP}. Evaluating the (continued) functional equation ~\eqref{eq:functional_equation} at $(\theta_1,\Theta_2^\pm(\theta_1))$, we obtain
\begin{equation*}
     \psi_1(\Theta_2^\pm(\theta_1))+\psi_2(\theta_1)=0,
\end{equation*}
which immediately implies that
\begin{equation}
\label{eq:before_boundary_condition}
     \psi_1(\Theta_2^+(\theta_1))=\psi_1(\Theta_2^-(\theta_1)).
\end{equation}
Choosing $\theta_1\in (-\infty,\theta_1^-)$, the two quantities $\Theta_2^+(\theta_1)$ and $\Theta_2^-(\theta_1)$ are complex conjugate the one of the other, see Section~\ref{sec:analytic}. Equation~\eqref{eq:before_boundary_condition} can then be reformulated as~\eqref{eq:boundary_condition}, using the definition~\eqref{eq:curve_definition} of the curve $\R$.
\end{proof}

\subsection{Conformal gluing function and invariant theorem}
The BVP stated in Proposition~\ref{prop:BVP} is called a homogeneous BVP with shift (the shift stands here for the complex conjugation, but the theory applies to more general shifts, see \cite{litvinchuk_solvability_2000}). Due to its particularly simple form, we can solve it in an explicit way, using the two following steps: 
\begin{itemize}
     \item Using a certain conformal mapping (to be introduced below), we can construct a particular solution to the BVP of Proposition~\ref{prop:BVP}.
     \item The solution to the BVP of Proposition~\ref{prop:BVP} is unique, see Lemma~\ref{lemma:invariant} below (taken from Lemma 2 in \cite[Section 10.2]{litvinchuk_solvability_2000}). In other words, two different solutions must coincide, and the explicit solution constructed above must be the function $\psi_1$.
\end{itemize}
\begin{lem}[Invariant lemma]
\label{lemma:invariant}
The problem of finding functions $f$ such that
\begin{enumerate}[label={\rm (\roman{*})},ref={\rm (\roman{*})}]
     \item\label{item:1_BVP_bis} $f$ is analytic in $\G$ and continuous in $\overline{\G}$,
     \item\label{item:2_BVP_bis} $f$ satisfies the boundary condition \eqref{eq:boundary_condition},
\end{enumerate}     
does not have non-trivial solutions in the class of functions $f$ vanishing at infinity.
\end{lem}

To construct a particular solution to the BVP of Proposition~\ref{prop:BVP}, we shall use the function
\begin{equation}
\label{eq:expression_w}
     {w} (\theta_2)=T_{\frac{\pi}{\beta}}\bigg(-\frac{2\theta_2-(\theta_2^++\theta_2^-)}{\theta_2^+-\theta_2^-}\bigg)
\end{equation}
introduced in Theorem~\ref{thm:main}. Let us first establish some of its properties.

\begin{lem}
\label{lem:CGF}
The function $w$ in \eqref{eq:expression_w} is such that:
\begin{enumerate}[label={\rm (\roman{*})},ref={\rm (\roman{*})}]
     \item\label{item:conformal_1} $w$ is analytic in $\mathcal{G}$, continuous in $\overline{\G}$ and unbounded at infinity,
     \item\label{item:conformal_2} $w$ is injective in $\mathcal{G}$ {\rm(}onto $\mathbb{C}\setminus(-\infty,-1]${\rm)},%\textcolor{red}{Je crois que c'est l'intervalle infini $[-i,-i\infty]$ ça se voit bien dans la formule avec $W$ de la preuve du lemme 6 (ou $[1,\infty]$ si j'ai tort pour le $\frac{-i}{2}$)}
     \item\label{item:conformal_3} $w(\theta_2)=w(\overline{\theta_2})$ for all $\theta_2\in\mathcal{R}$.
\end{enumerate}
\end{lem}

The function $w$ is called a conformal gluing function. The conformal property comes from~\ref{item:conformal_1} and~\ref{item:conformal_2}, and the gluing from~\ref{item:conformal_3}: $w$ glues together the upper and lower parts of the hyperbola $\R$. There are at least two ways for proving Lemma~\ref{lem:CGF}. First, it turns out that in the literature there exist expressions for conformal gluing functions for relatively simple curves: circles, ellipses, and also for hyperbolas, see \cite[Equation (4.6)]{Baccelli_Fayolle_1987}. Following this way, we could obtain the expression \eqref{eq:expression_w} for $w$ given above and its different properties stated in Lemma~\ref{lem:CGF}. Instead, we would like to use the Riemann sphere given by 
\begin{equation*}
     \{(\theta_1,\theta_2)\in\mathbb C^2 :\gamma(\theta_1,\theta_2)=0\}. 
\end{equation*}
Indeed, as we shall see in Section~\ref{sec:Riemann_surface}, many (not to say all) technical aspects, in particular finding the conformal mapping, happen to be quite simpler on that surface. The proof of Lemma~\ref{lem:CGF} is thus postponed to Section~\ref{sec:Riemann_surface}.

Under the symmetry condition ($\mu_1=\mu_2$, $\sigma_{11}=\sigma_{22}$ and symmetric reflection vectors $R_1,R_2$ in~\eqref{eq:RBMQP}), Foschini also obtained an expression for that conformal mapping, see \cite[Figure 3]{Foschini}.

\subsection{Proof of Theorem~\ref{thm:main}} 

We are now ready to prove our main result. Let us introduce the function
\begin{equation*}
     f(\theta_2)=\psi_1(\theta_2) - \phi_1(0)\frac{w'(0)}{w(\theta_2)-w(0)}.
\end{equation*}
The key point is that $f$ satisfies the assumptions of Lemma~\ref{lemma:invariant}. 

The first item~\ref{item:1_BVP_bis} of Lemma~\ref{lemma:invariant} holds by construction: the only possible pole of $f$ is at $0$ ($\psi_1$ has a unique pole at $0$, see Proposition~\ref{prop:BVP}, and since $w$ is injective, see Lemma~\ref{lem:CGF}, the equation $w(\theta_2)-w(0)=0$ has only one solution, viz, $\theta_2=0$). However, a series expansion shows that the residue of $f$ at $0$ is $0$, in other words $0$ is a removable singularity. 

Point~\ref{item:2_BVP_bis} is also clear, since both $\psi_1$ and $w$ satisfy the boundary condition \eqref{eq:boundary_condition}. Furthermore, $f$ vanishes at infinity, since on the one hand $\psi_1$ does, and on the other hand $w$ goes to infinity at infinity. Using Lemma~\ref{lemma:invariant}, we conclude that $f=0$. It remains to show that $\phi_1(0)=-\mu_1$. For this it is enough to evaluate the functional equation~\eqref{eq:functional_equation}, first at $\theta_2=0$, and then at $\theta_1=0$.

\subsection{Diagonal covariance and dimension one}
\label{subsec:diagonal_covariance}
Assuming that the covariance matrix $\Sigma$ is diagonal, we have $\beta=\frac{\pi}{2}$ in \eqref{eq:definition_beta}, and $T_{\frac{\pi}{\beta}}=T_2$ is the second Chebyshev polynomial, given by $T_2(x)=2x^2-1$. The formula \eqref{eq:expression_w} for $w$ together with the expression \eqref{eq:definition_theta_2_pm} of $\theta_2^\pm$ yields
\begin{equation*}
     {w}(\theta_2)-w(0)=\frac{8}{\mu_1^2\sigma_{11}\sigma_{22}+\mu_2^2\sigma_{22}^2}\theta_2\left(\theta_2+\frac{2\mu_2}{\sigma_{22}}\right).
\end{equation*}
The formula of Theorem~\ref{thm:main} then gives 
\begin{equation*}
     \phi_1(\theta_2)=\frac{-\mu_1 w'(0)}{{w}(\theta_2)-{w}(0)} \theta_2=-\frac{2\mu_1\mu_2}{\sigma_{22}} \frac{1}{\theta_2+\frac{2\mu_2}{\sigma_{22}}},
\end{equation*}
and finally, after some elementary computations and the use of functional equation~\eqref{eq:functional_equation}, we obtain
\begin{equation}
\label{eq:expression_phi_diagonal_cov}
     \phi(\theta)=\frac{2\mu_1/\sigma_{11}}{(\theta_1+{2\mu_1}/{\sigma_{11}})}\frac{2\mu_2/\sigma_{22}}{(\theta_2+{2\mu_2}/{\sigma_{22}})}.
\end{equation}
By inversion of the Laplace transform, we reach the conclusion that
\begin{equation*}
     \nu_1 (x_2)=\frac{2\mu_1\mu_2}{\sigma_{22}} e^{\frac{2\mu_2}{\sigma_{22}} x_2}, \qquad 
     \nu_2 (x_1)=\frac{2\mu_1\mu_2}{\sigma_{11}} e^{\frac{2\mu_1}{\sigma_{11}} x_1},  \qquad
     \pi(x_1,x_2)=\frac{4\mu_1\mu_2}{\sigma_{11}\sigma_{22}}e^{\frac{2\mu_1}{\sigma_{11}} x_1+\frac{2\mu_2}{\sigma_{22}}x_2}.
\end{equation*}

Let us do two comments around \eqref{eq:expression_phi_diagonal_cov}. Firstly, the product-form expression \eqref{eq:expression_phi_diagonal_cov} is easily explained by the skew-symmetric condition (see \cite[Equation (10.2)]{MR628959})
\begin{equation*}
     2\Sigma=R\cdot\text{diag}(R)^{-1} \cdot \text{diag}(\Sigma)
+\text{diag}(\Sigma) \cdot \text{diag}(R)^{-1}\cdot R^{\top},
\end{equation*}
which is obviously satisfied if $R$ is the identity matrix and $\Sigma$ is diagonal (above, $\text{diag}(R)$ denotes the diagonal matrix with same diagonal entries as $R$ and $R^\top$ is the transpose matrix of $R$). 
Secondly, our method works also to study the one-dimensional case. Let indeed
\begin{equation*}
     X_t=X_0+W_t+\mu t + L_t^0
\end{equation*}
be a one-dimensional reflected Brownian motion with drift $\mu$, where $W$ is a Brownian motion of variance $\sigma>0$ and $L_t^0$ is the local time at $0$. Applying It\^o formula to $(e^{\theta X_t})$ and taking expected value over the invariant measure $\pi$, we obtain, with $\phi(\theta)=\mathbb{E}_\pi [e^{\theta X_1}]$,  
\begin{equation*}
     \left(\frac{\sigma}{2}\theta+\mu\right)\phi (\theta)=- \mathbb{E}_\pi[ L_1^0].
\end{equation*}
This is the functional equation in dimension one (compare with \eqref{eq:functional_equation}). Evaluating this identity at $0$ and remembering that $\phi(0)=1$ we find
$\mathbb{E}_\pi [L_1^0]=-\mu$, and then
\begin{equation*}
     \phi(\theta)=\frac{2\mu/\sigma}{\theta+2\mu/\sigma}.
\end{equation*}
This coincides with \eqref{eq:expression_phi_diagonal_cov}.

\subsection{Statement of the BVP in the general case}
\label{subsec:general_case}

We would like to close Section~\ref{sec:BVP} by stating the BVP in the case of arbitrary reflections (non necessarily orthogonal). Let us define for $\theta_2\in \R$
\begin{equation*}
     G(\theta_2)=\frac{\gamma_1}{\gamma_2}(\Theta_1^-(\theta_2),\theta_2)\frac{\gamma_2}{\gamma_1}(\Theta_1^-(\theta_2),\overline{\theta_2}).
\end{equation*}     
Using the same line of arguments as in the proof of Proposition~\ref{prop:BVP}, we obtain the following result:

\begin{prop}
\label{prop:BVP_general}
The function $\phi_1$ satisfies the following BVP:
\begin{enumerate}[label={\rm (\roman{*})},ref={\rm (\roman{*})}]
     \item\label{item:1_BVP_general} $\phi_1$ is meromorphic on $\G$ with at most one pole $p$ of order $1$, and is bounded at infinity, %\textcolor{red}{C'était pas très clair de la manière dont je l'avais écrit mais je voulais dire qu'il y avait au plus un pole (qui est $\theta^*$ le zéro de $\gamma_1$ si il est avant $\theta_1^-$)}
     \item\label{item:2_BVP_general} $\phi_1$ is continuous on $\overline{\G}\setminus \{p\}$ and
\begin{equation}
\label{eq:boundary_condition_general}
     \phi_1(\overline{\theta_2})=G(\theta_2)\phi_1({\theta_2}), \qquad \forall \theta_2\in \R.
\end{equation}
\end{enumerate}
\end{prop}

Due to the presence of the function $G\neq1$ in \eqref{eq:boundary_condition_general}, this BVP (still homogeneous with shift) is more complicated than the one encountered in Proposition~\ref{prop:BVP}, and cannot be solved thanks to an invariant lemma. Instead, the resolution is less combinatorial and far more technical, and the solution should be expressed in terms of Cauchy integrals and the conformal mapping $w$ of Lemma~\ref{lem:CGF}. This will be achieved in a future work. 

From the viewpoint of the general BVP of Proposition~\ref{prop:BVP_general}, we thus solved in this note a case where the variables $\theta_1$ and $\theta_2$ could be separated in the quantity (the right-hand side of the functional equation \eqref{eq:functional_equation})
\begin{equation*}
     \gamma_1 (\theta) \varphi_1 (\theta_2) + \gamma_2 (\theta) \varphi_2 (\theta_1).
\end{equation*}
Let us say a few words on another case (the only other case, in fact) where the variables can also be separated, namely, when $\gamma_1=\gamma_2$. In this case the reflections are parallel, and the function $\phi_1$ (instead of $\psi_1$) satisfies the BVP of Proposition~\ref{prop:BVP}, with no pole in $\G$. With Lemma~\ref{lemma:invariant}, it has to be a constant. In other words, there is no invariant measure, which is in accordance with the fact that with parallel reflections the condition of ergodicity \eqref{eq:CNS_ergodic} is obviously not satisfied.

\section{Singularity analysis}
\label{sec:singularity}

\subsection{Statement of the result}

In the literature, an important aspect of reflected Brownian motion in the quarter plane (and more generally in orthants) is the asymptotics of its stationary distribution, see indeed \cite{dai_reflecting_2011,franceschi_asymptotic_2016} for the asymptotics of the interior measure, and \cite{dai_stationary_2013} for the boundary measures. With our Theorem~\ref{thm:main} and a classical singularity analysis of Laplace transforms (our main reference for this is the book \cite{Doetsch}), we can easily obtain such asymptotic expansions.

Precisely, we identify three regimes, depending on the sign of $\Theta_1(\theta_2^+)$. The fact that the latter quantity determines the asymptotics can be easily explained: as we shall see in Section~\ref{sec:Riemann_surface}, the pole $\theta_2\in\mathbb{C} \setminus [\theta_2^+,\infty)$ of $\phi_1$ that is the closest to the origin is (provided it exists) a solution to $\Theta_1^+(\theta_2)=0$. The relative locations of $\Theta_1^+(\theta_2)$ and $0$ will therefore decide which of that pole or of the algebraic singularity $\theta_2^+$ will be the first singularity of $\phi_1$. 

Define the constants
\begin{equation}
\label{eq:expression_C_1_C_2}
     C_1= \frac{-\phi_1(\theta_2^+) 2\frac{\pi}{\beta}\sin(\frac{\pi}{\beta}\pi)}{(w(\theta_2^+)-w(0))\sqrt{\theta_2^+-\theta_2^-}},\qquad C_2=\frac{-\mu_1 w'(0)\theta_2^+ \sqrt{\theta_2^+-\theta_2^-}}{2\frac{\pi}{\beta}\sin(\frac{\pi}{\beta}\pi)}.
\end{equation}
%$=-\frac{\sigma_{12}\theta_2^++\mu_1}{\sigma_{11}}$. 
%This condition admits a simple geometric interpretation. Let $\beta$ be the angle defined in \eqref{eq:definition_beta} and $\alpha\in(0,\pi)$ as in
%\begin{equation*}
%     \alpha=\frac{1}{2}\arg \left(\frac{s_0}{s_0'}\right),
%\end{equation*}
%see Figure~\ref{fig:correspondance_plane_sphere} (where $s_0=e^{i(\pi+\alpha)}$ and $s_0'=e^{i(\pi-\alpha)}$). Then the sign of $\Theta_1(\theta_2^+)$ is the same as the sign of $\pi-(\alpha+\beta)$. 
\begin{prop}
\label{prop:asymptotics}
As $x_2\to\infty$, 
%if $\pi/\beta\notin\mathbb Z$ 
the asymptotics of $\nu_1 (x_2)$ is given by
\begin{equation*}
     \nu_1 (x_2)= \left\{\begin{array}{ll}
     \displaystyle {x_2}^{-\frac{3}{2}}e^{-\theta_2^+x_2} \left(\frac{-C_1}{2\sqrt{\pi}}+o(1)\right)& \text{if $\Theta_1(\theta_2^+)<0$},\smallskip
     \\
     \displaystyle {x_2}^{-\frac{1}{2}}e^{-\theta_2^+x_2} \left(\frac{C_2}{\sqrt{\pi}}+o(1)\right) & \text{if $\Theta_1(\theta_2^+)=0$},\smallskip
     \\
     \displaystyle e^{2\frac{\mu_2}{\sigma_{22}}x_2}\left(\frac{2\mu_1\mu_2}{\sigma_{22}}+o(1)\right)  & \text{if $\Theta_1(\theta_2^+)>0$},
     \end{array}\right.
\end{equation*}   
where the constants $C_1$ and $C_2$ are defined in \eqref{eq:expression_C_1_C_2}.
%If $\pi/\beta \in\mathbb Z$, the asymptotics is 
%\begin{equation*}
%     \nu_1 (x_2)=e^{2\frac{\mu_2}{\sigma_{22}}x_2}\left(\frac{2\mu_1\mu_2}{\sigma_{22}}+o(1)\right).
%\end{equation*}     
\end{prop}

Note that we obtain (for the particular case of orthogonal reflections) the asymptotic result of Dai and Miyazawa in \cite[Theorem 6.1]{dai_reflecting_2011}, with the additional information of the explicit expression for the constants. 

Note that $\beta=\pi/k$ for $k\geq 2$ implies $\Theta_1(\theta_2^+)>0$ and that $k$ cannot be $1$ due to the condition on $\Sigma$.

%\textcolor{red}{De plus que se passe-t-il si le sinus vaut $0$? C'est possible apres tout, des que $\beta=\pi/k$ avec $k\in\mathbb N$. OUI en fait c'est possible, c'est le cas ou $w$ est  un polynome. Du coup c'est normal que $C_1$ s'annule car il n'y a plus de singularité dans le lemme 9. C'est un peu plus bizarre que $C_2$ soit infini même si c'est "normal". Dans le deuxième cas du lemme 9 c'est en fait un pôle simple du coup. Du coup quand ça arrive on est toujours dans le cas ou c'est le pôle qui domine donc dans la dernière asymptotique. En fait $\beta=\pi/k$ implique que $\Theta_1(\theta_2^+)>0$, i.e. $1/s_0$ après $e^{i\beta}$ sur le cercle, ça se voit bien pour le dessin , car $\beta$ est plus petit que $\pi/2$. Il n'y a que le cas où $\beta=\pi$ mais c'est un cas très particulier qui est celui où $\Theta_1(\theta_2^+)=0$, la il faudrait préciser le résultat avec $C_2$ qui est infini. J'avais proposé une nouvelle formulation de la proposition 8 pour éviter toutes ces explications mais je ne suis pas sur qu'elle soit necessaire, car il n'y a que le cas $\beta=\pi$ qui pose réellement un petit problème.} 

\subsection{Proof of Proposition~\ref{prop:asymptotics}}

Proposition~\ref{prop:asymptotics} is an easy consequence of the singularity analysis of $\phi_1$ and of classical transfer theorems, as \cite[Theorem 37.1]{Doetsch}. Due to the expression of $\phi_1$ in Theorem~\ref{thm:main}, there are two sources of singularities: the singularities of $w$ and the points $\theta_2$ where the denominator $w(\theta_2)-w(0)$ of $\phi_1$ vanishes.

Let us first study the singularities of $w$. In fact, the function $w$ can not only be analytically continued on $\G$ as claimed in Lemma~\ref{lem:CGF}, but on the whole of the cut plane $\mathbb C\setminus [\theta_2^+,\infty)$. Further, except if $\pi/\beta\in\mathbb Z$, in which case $w$ is a polynomial, $w$ has an algebraic-type singularity at $\theta_2^+$, given by the following result:

%(see Lemma~\ref{lem:nature_solution} for further considerations on the algebraic nature of $\phi_1$).

\begin{lem}
\label{lem:singularities_1}
If $\pi/\beta\notin\mathbb Z$, $\phi_1$ has an algebraic-type singularity at $\theta_2^+$, in the neighborhood of which it admits the expansion: %\textcolor{red}{J'ai calculé la constante qui découle juste d'un DL de $w(\theta_2)=\cos(\frac{\pi}{\beta}\pi)+2\frac{\pi}{\beta}\frac{\sin(\frac{\pi}{\beta}\pi)}{\sqrt{\theta_2^+-\theta_2^-}}\sqrt{\theta_2^+-\theta_2} +O({\theta_2-\theta_2^+})$. J'ai aussi rajouté le cas ou le pole "coïncide" avec la singularité (ce qui nous donne dans la proposition précédente le cas où $\Theta_1(\theta_2^+)=0$. Les constantes ne sont pas très jolies. Les conditions sont très simples maintenant mais sortent un peu du chapeau si on ne connait pas la sphere ou l'ellipse}
\begin{equation*}
     \phi_1(\theta_2)=
     \begin{cases}    
     \phi_1(\theta_2^+) +
   C_1 \sqrt{\theta_2^+-\theta_2}  +O({\theta_2-\theta_2^+}) & \text{ if } 
   %\Theta_1(\theta_2^+) \neq 0, 
   w(\theta_2^+)-w(0)\neq 0,
   \\
   \frac{C_2}{\sqrt{\theta_2^+-\theta_2}} +O(1) & \text{ if }
   %\Theta_1(\theta_2^+)=0, 
   w(\theta_2^+)-w(0)= 0,
   \end{cases}
\end{equation*}
where $C_1$ and $C_2$ are defined in \eqref{eq:expression_C_1_C_2}.   
\end{lem}

\begin{proof}
 %\textcolor{red}{ We first notice that $w(\theta_2^+)=w(0)$ if and only if $\Theta_1(\theta_2^+)=0$ Sandro, c'est bien ca ? En fait j'ai remis les conditions que j'avais mis au début, j'avais trop simplifié. En fait on a $\Theta_1(\theta_2^+)=0$ implique $w(\theta_2^+)=w(0)$ mais la réciproque est fausse car $\theta_2^+$ peut par exemple coïncider avec le "deuxième pôle" qui ne vérifie pas $\Theta_1(\theta_2)=0$ mais qui de toute façon ne donnera pas l'asymptotique. Du coup la proposition n'était pas vrai dans certain cas donc j'ai re changé les conditions mais ça n'impacte pas trop les autres résultats}. 
 Lemma~\ref{lem:singularities_1} follows from doing an expansion of $\phi_1$ at $\theta_2^+$, using (the elementary) Lemma~\ref{lem:Chebyshev_singularity} below.
\end{proof}

\begin{lem}
\label{lem:Chebyshev_singularity}
If $a$ is an integer, the generalized Chebyshev polynomial $T_a$ is the classical Chebyshev polynomial. If not, $T_a$ is not a polynomial, and admits an analytic extension on $\mathbb{C}\setminus (-\infty,-1]$. The point $-1$ is an algebraic-type singularity, and there is the expansion
\begin{equation}
\label{eq:expansion_Chebyshev_1}
     T_a(x)=\cos(a\pi)+a\sqrt{2}\sin(a\pi)\sqrt{x+1}+O({x+1}).
\end{equation}
%\textcolor{red}{C'est $-1$ du coup pour être cohérent avec la détermination principale du log}
\end{lem}
\begin{proof}
The considerations on the algebraic nature of $T_a$ are clear, and \eqref{eq:expansion_Chebyshev_1} comes from making an expansion of $T_a(x)=\cos(a\arccos(x))$ in the neighborhood of $-1$.
\end{proof}

We now turn to the singularities introduced by the denominator of $\phi_1$; these singularities are poles.

\begin{lem}
\label{lem:singularities_2}
There are two cases concerning the poles of $\phi_1$ in the cut plane $\mathbb{C}\setminus [\theta_2^+,\infty)$:
\begin{itemize}
     \item $\Theta_1(\theta_2^+)\leqslant 0$, 
     then $\phi_1$ is analytic on $\mathbb{C}\setminus [\theta_2^+,\infty)$,
     \item $\Theta_1(\theta_2^+)>0$, %if $\alpha+\beta <\pi$, 
     then $\phi_1$ is meromorphic on $\mathbb{C}\setminus [\theta_2^+,\infty)$, with poles on the segment $[0,\theta_2^+]$ only. The closest pole to the origin is at $-2\frac{\mu_2}{\sigma_{22}}$ and has order one.
\end{itemize}
\end{lem}
The proof of Lemma~\ref{lem:singularities_2} is postponed to Section~\ref{sec:Riemann_surface}, since the tools that we shall introduce there will simplify it.
Then Proposition~\ref{prop:asymptotics} is an easy consequence of Lemmas~\ref{lem:singularities_1} and~\ref{lem:singularities_2}.

%\textcolor{red}{En fait j'ai un problème dans les conditions. pour moi $\Theta_1(\theta_2^+)=-\frac{\sigma_{12}\theta_2^++\mu_1}{\sigma_{11}}=0$ devrait impliquer que $\theta_2^+$ est aussi le pôle i.e. $\theta_2^+=-\frac{2\mu_2}{\sigma_{22}}$ mais je ne le retrouve pas dans le calcul. Je continuerai à chercher demain j'ai peut être juste raté un truc mais je ne comprend pas trop... C'est pour ça que je n'ai pas encore vraiment fini de bien écrire les différentes conditions avec ou sans les "conditions géométriques".}

%T_1(x)= \cos\left(\frac{\pi a}{2}\right)+ a \sum_{j\geqslant 1}\frac{(2x)^j}{2j}\cos\left(\frac{\pi(a-j)}{2}\right){\frac{a+j-2}{2} \choose j-1}$$, where ${b \choose k}=\frac{\prod_{i=0}^{k-1}(b-i)}{k!}$. 

\section{Riemann sphere and related facts}
\label{sec:Riemann_surface}

We have many objectives in this section, which all concern the set of zeros of the kernel
\begin{equation*}
     \s=\{(\theta_1,\theta_2)\in\mathbb{C}^2 : \gamma(\theta_1,\theta_2)=0\}.
\end{equation*}
We add some complexity here, by using the framework of Riemann surfaces. In return, many technical aspects become more intrinsic and some key quantities admit nice and natural interpretations. We first (Section~\ref{subsec:param}) study the structure of $\s$, as a Riemann surface. Then we find a simple formula for the conformal mapping $w$ (Section~\ref{subsec:conformal_mapping}). We finally introduce the notion of group of the model (Section~\ref{subsec:group}), similar to the notion of group of the walk in the discrete setting \cite{Maly,fayolle_random_1999,BMMi-10}, and we prove that the algebraic nature of the solution is related to the finiteness of this group.

\subsection{Uniformisation}
\label{subsec:param}

Due to the degree of the kernel $\gamma$, the surface $\s$ has genus $0$ and is a Riemann sphere $\mathbb{C}\cup\{\infty\}$, see \cite{franceschi_asymptotic_2016}. It thus admits a rational parametrization (or uniformisation) $\{(\theta_1(s),\theta_2(s)): s\in\s\}$, given by
\begin{equation}
\label{eq:formulas_uniformization}
     \left\{\begin{array}{rcl}
     \displaystyle \theta_1(s)&=&\displaystyle\frac{\theta_1^{+}+\theta_1^{-}}{2}+\frac{\theta_1^{+}-\theta_1^{-}}{4} \left(s+\frac{1}{s}\right),\medskip\\
     \displaystyle \theta_2(s)&=&\displaystyle\frac{\theta_2^{+}+\theta_2^{-}}{2}+\frac{\theta_2^{+}-\theta_2^{-}}{4} \left(\frac{s}{e^{i\beta}}+\frac{e^{i\beta}}{s}\right),
     \end{array}\right.
\end{equation}
with $\beta$ as in \eqref{eq:definition_beta}. The equation $\gamma(\theta_1(s),\theta_2(s))=0$ is valid for any $s\in\s$. We will often represent a point $s\in\s$ by the pair of coordinates $(\theta_1(s),\theta_2(s))$. 

Any point $\theta_1\in\mathbb C$ has two images on $\s$. More specifically, let $s\in\s$ be defined by $\theta_1(s)=\theta_1$. Then the two points are given by $(\theta_1(s),\theta_2(s))$ and $(\theta_1(s),\theta_2(1/s))$. They are always different, except if $\theta_1$ is a branch point $\theta_1^\pm$. Similarly, any $\theta_2\in\mathbb C$ corresponds to two points on $\s$, viz, $(\theta_1(s),\theta_2(s))$ and $(\theta_1(e^{2i\beta}/s),\theta_2(s))$.

Any point or domain of $\mathbb C$ can be represented on the Riemann sphere, and there is the following correspondance (which is easily proved by using the formulas \eqref{eq:formulas_uniformization}), see Figure~\ref{fig:correspondance_plane_sphere}: 
\begin{itemize}
     \item The branch points $\theta_1^\pm$ and $\theta_2^\pm$ are located at $\pm 1$ and $\pm e^{i\beta}$;
     \item The point of coordinates $(0,0)$ is at $s_0$, %$s_0'$ and $s_0''$,
      and $\infty$ is at $0$ and $\infty$;
     \item The real points $(\theta_1,\theta_2)\subset \mathbb R^2$ such that $\gamma(\theta_1,\theta_2)=0$ form the unit circle;
     \item The curve $\R$ is sent to $(-\infty,0)$, and if $\theta_2\in \R$ corresponds to $s\in (-\infty,0)$, then $\overline{\theta_2}\in\R$ corresponds to $1/s$; the domain $\G$ is the cone bounded by $(0,-\infty)$ and $(0,-e^{i\beta}\infty)$.
\end{itemize}

\begin{figure}[hbtp]
\centering
\includegraphics[scale=1]{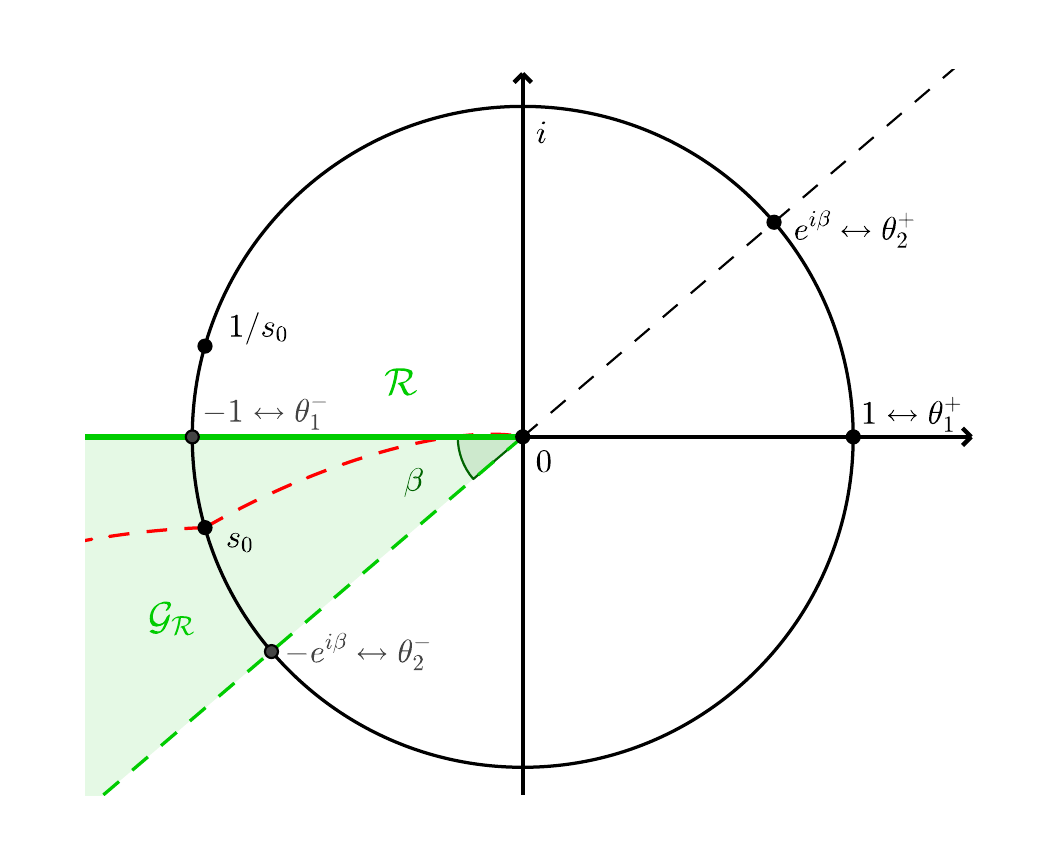}
\caption{The set $\s$ viewed as the Riemann sphere, and important points and curves on it 
%\textcolor{red}
%{Je ne pense pas que ce soit la peine d'avoir deux dessins, un seul suffit. Mais du coup il faut bien verifier que toutes les informations necessaires y sont.}
}
\label{fig:correspondance_plane_sphere}
\end{figure}

%\begin{figure}[hbtp]
%\centering 
%\includegraphics[scale=0.9]{images/ellipse.pdf} 
%\caption{Ellipse $\mathcal{E}$.}
% \label{ellipse}
%\end{figure}

\subsection{Conformal mapping}
\label{subsec:conformal_mapping}

In this section, we first show how to obtain the expression \eqref{eq:expression_w} of $w$ in terms of generalized Chebyshev polynomials, and we prove Lemma~\ref{lem:CGF}. Let us first notice that any function $f(\theta_2)$ can be lifted in a function of $F(s)$, by mean of the formula
\begin{equation*}
     F(s)=f(\theta_2(s)).
\end{equation*}
In particular, $W(s)$ stands for the lifted conformal mapping $w(\theta_2(s))$. Reciprocally, for a function $F(s)$ to define a univalent function $f(\theta_2)$, the condition $F(s)=F({e^{2i\beta}}/{s})$ needs to hold, due to the obvious identity $\theta_2(s)=\theta_2(e^{2i\beta}/{s})$.

\begin{proof}[Proof of Lemma~\ref{lem:CGF}]
We first translate on $W$ the properties of Lemma~\ref{lem:CGF} stated for $w$. First, $W$ has to be analytic in the open cone $\G$ (in $\s$, this is the cone delimitated by $(0,-\infty)$ and $(0,-e^{i\beta}\infty)$) and continuous on the closed cone $\overline{\G}$ except at $0$, see~\ref{item:conformal_1}. Second, $W$ has to be injective on $\G$, see~\ref{item:conformal_2}. Finally, the boundary condition~\ref{item:conformal_3} has to be replaced by the pair of conditions
\begin{equation}
\label{eq:boundary_conditions_CGF}
     \left\{\begin{array}{ll}
     W(s)=W({1}/{s}), & \forall s\in\R,\\
     W(s)=W({e^{2i\beta}}/{s}), & \forall s\in(0,-e^{i\beta}\infty).
     \end{array}\right.
\end{equation}
We easily come up with the (at this point: conjectured) formula
\begin{equation}
\label{eq:expression_W}
     W(s)=-\frac{1}{2}\big\{(-s)^{\frac{\pi}{\beta}}+(-s)^{-\frac{\pi}{\beta}}\big\}=-\frac{1}{2}\big\{e^{\frac{\pi}{\beta}\log (-s)}+e^{-\frac{\pi}{\beta}\log (-s)}\big\},
\end{equation}
where we make use of the principal determination of the logarithm. (Since the conditions \eqref{eq:boundary_conditions_CGF} are invariant under multiplication by a constant, we can choose the constant in front of the right-hand side of \eqref{eq:expression_W}. We choose $-1/2$, so as to match the expression $W(s)$ with $w(\theta_2(s))$, $w$ being as in Theorem~\ref{thm:main}: $w(\Theta_2(\theta_1^-))=T_{\frac{\pi}{\beta}} (\cos \beta ) =\cos \pi=-1= W(-1)$.)
%on the set $\s\setminus \mathcal{R}_{\theta_2^+}=\mathbb{C}\setminus e^{i\beta}\mathbb{R}$ (we may chose for example $\log(s_1^-)=\log(-1)=i\pi$ and then $w(s_1^-)=2\cos(\frac{\pi^2}{\beta})$).

Let us briefly verify each property. First, the analyticity property~\ref{item:conformal_1} is clear from the properties of the logarithm. In order to show~\ref{item:conformal_2}, we first remark that $W(s)=W(t)$ if and only if for some $k\in\mathbb Z$,
\begin{equation}
\label{eq:resolution_WW}
     s=t^{\pm1} e^{2ik\beta}.
\end{equation}
Since $s$ and $t$ both belong to the cone $\G$, we must have $s=t$, and therefore $W$ is injective. Finally,~\ref{item:conformal_3} is clear from the construction of $W$.

To obtain the expression of $w$ in terms of the generalized Chebyshev polynomial, we use the fact that $\theta_2(s)=\theta_2$ yields the formula
\begin{equation*}
     s=\frac{e^{i\beta}}{\theta_2^+-\theta_2^-} \left(2\theta_2-(\theta_2^++\theta_2^-)\pm2\sqrt{(\theta_2^+-\theta_2)(\theta_2^--\theta_2)}\right).
\end{equation*}
The proof is complete.
\end{proof}

Thanks to the expression \eqref{eq:expression_W} of $W$ derived above, we can now prove~Lemma~\ref{lem:singularities_2}, concerning the (eventual) poles of $\phi_1$ in the cut plane $\mathbb{C}\setminus [\theta_2^+,\infty)$.

\begin{proof}[Proof of Lemma~\ref{lem:singularities_2}]
Using the expression of $\phi_1$ obtained in Theorem~\ref{thm:main}, we conclude that the poles of $\phi_1$ must be located at points $\theta_2$ where $w(\theta_2)=w(0)$. Using the function $W$ in \eqref{eq:expression_W}, we reformulate the latter identity as $W(s)=W(s_0)$, see Figure~\ref{fig:correspondance_plane_sphere}, and thanks to \eqref{eq:resolution_WW} we obtain $s=(s_0)^{\pm 1} e^{2ik\beta}$. The closest pole to the origin $s_0$ is then ${1}/{s_0}$, which corresponds to the point 
\begin{equation*}
     \theta_2=-2\frac{\mu_2}{\sigma_{22}} 
\end{equation*}
such that $\Theta_1^+(\theta_2)=0$, since $\theta_1(s_0)=\theta_1(1/s_0)=0$. The sign of $\Theta_1(\theta_2^+)$ then determines if ${1}/{s_0}$ is before or after $e^{i\beta}$ (which corresponds to $\theta_2^+$) on the unit circle, see Figure~\ref{fig:correspondance_plane_sphere}. The conditions given in terms of sign of $\Theta_1(\theta_2^+)$ follows. Finally, let us notice that $w(\theta_2^+)=w(0)$ if $\Theta_1(\theta_2^+)=0$ and that $w(\theta_2^+)\neq w(0)$ if $\Theta_1(\theta_2^+)<0$.
\end{proof}

\subsection{Group of the model and nature of the solution}
\label{subsec:group}

The notion of group of the model has been introduced by Malyshev \cite{Maly} in the context of random walks in the quarter plane. It turns out to be an important characteristic of the model, in particular to decide whether generating functions or Laplace transforms are algebraic or D-finite functions, see \cite{BMMi-10}. 

It can be introduced directly on the kernel $\gamma$: with the notation \eqref{eq:definition_a_b_c}, this is the group $\langle\zeta,\eta\rangle$ generated by $\zeta$ and $\eta$, given by 
\begin{equation*}
     \zeta(\theta_1,\theta_2)=\left(\theta_1,\frac{c(\theta_1)}{a(\theta_1)}\frac{1}{\theta_2}\right),\qquad
     \eta(\theta_1,\theta_2)=\left(\frac{\widetilde c(\theta_2)}{\widetilde a(\theta_2)}\frac{1}{\theta_1},\theta_2\right).
\end{equation*}
By construction, the generators satisfy $\gamma(\zeta(\theta_1,\theta_2))=\gamma(\eta(\theta_1,\theta_2))=0$ as soon as $\gamma(\theta_1,\theta_2)=0$. In other words, there are (covering) automorphisms of the surface $\s$. Since $\zeta^2=\eta^2=1$, the group $\langle\zeta,\eta\rangle$ is a dihedral group, which is finite if and only if the element $\zeta\eta$ (or $\eta\zeta$) has finite order. 

With the above definition, it is not clear how to see if the group is finite, nor to see if this has any implication on the problem. In fact, we have:
\begin{prop}
\label{prop:group_finite}
The group $\langle\zeta,\eta\rangle$ is finite if and only if $\pi/\beta\in\mathbb Q$.
\end{prop}
The proof of Proposition~\ref{prop:group_finite} is simple, once the elements $\zeta$ and $\eta$ have been reformulated on the sphere $\s$:
\begin{equation*}
     \zeta(s)=\frac{1}{s},\qquad \eta(s)=\frac{e^{2i\beta}}{s}.
\end{equation*}
These transformations leave invariant $\theta_1(s)$ and $\theta_2(s)$, respectively, see \eqref{eq:formulas_uniformization}. In particular, we have the following result (consequence of the Lemma~\ref{lem:nature_solution}), which connects the nature of the solution of the BVP to the finiteness of the group. Such a result holds for discrete walks, see \cite{BMMi-10,BeBMRa_2015}.
\begin{prop}
\label{prop:nature}
The solution $\phi_1$ given in Theorem~\ref{thm:main} is algebraic if and only if the group $\langle\zeta,\eta\rangle$ is finite.
\end{prop}
The proof of Proposition~\ref{prop:nature} builds on the following elementary result:
\begin{lem} 
\label{lem:nature_solution}
Let $a\geq0$. The generalized Chebyshev polynomial $T_a$ is
\begin{itemize}
     \item rational if $a\in\mathbb Z$,
     \item algebraic (and not polynomial) if $a\in \mathbb{Q}\setminus\mathbb{Z}$,
     \item not algebraic if $a\in \mathbb{R}\setminus \mathbb{Q}$.
\end{itemize}
\end{lem}
However, even in the non-algebraic case, the Chebyshev polynomial $T_a$ always satisfies a linear differential equation with coefficients in $\mathbb R$, since it can be written as the particular hypergeometric function ${}_2F_1([-a,a],[1/2],1-x)$, where (with $(A)_k=A(A+1)\cdots (A+(k-1))$)
\begin{equation*}
     {}_2F_1([A,B],[C],t) = \sum_{k=0}^{\infty}\frac{(A)_k (B)_k}{(C)_k}\frac{t^k}{n!}.
\end{equation*}

%\tableofcontents

\nocite{BuChMaRa_2015,Foschini,EA-AofA2016}
\bibliographystyle{apalike} % Le style est mis entre accolades.
\bibliography{biblio} % mon fichier de base de données s'appelle bibli.bib

\end{document}